\documentclass{birkjour}
\usepackage[utf8]{inputenc}
\usepackage[T1]{fontenc}
\usepackage[english]{babel}
\usepackage{mathtools,amsmath,amssymb,amsfonts,amsthm,mathrsfs}
\mathtoolsset{showonlyrefs}

\usepackage{enumitem}

\usepackage{geometry}
\usepackage{xcolor}

\usepackage[pdftex]{hyperref}

\hypersetup{
  bookmarksnumbered, linkcolor=gray, citecolor=gray,
  bookmarks, breaklinks,
  linkbordercolor=gray,
  citebordercolor=gray,
}

\newtheorem{theorem}{Theorem}[section]

\newtheorem{proposition}{Proposition}[section]

\theoremstyle{definition}
\newtheorem{remark}{Remark}[section]

\numberwithin{equation}{section}

\begin{document}

\title[Continuation theorems for periodic systems]{Continuation theorems for periodic systems\\ and applications to problems with nonlinear \\time-dependent differential operators}

\author[P.~Benevieri]{Pierluigi Benevieri}

\address{
Instituto de Matem\'{a}tica e Estat\`{i}stica\\
Universidade de S\~{a}o Paulo\\
Rua do Mat\~{a}o 1010, S\~{a}o Paulo, SP - CEP 05508-090, Brazil}

\email{pluigi@ime.usp.br}

\author[G.~Feltrin]{Guglielmo Feltrin}

\address{
Dipartimento di Scienze Matematiche, Informatiche e Fisiche\\
Universit\`{a} degli Studi di Udine\\
Via delle Scienze 206, 33100 Udine, Italy}

\email{guglielmo.feltrin@uniud.it}

\thanks{Work written under the auspices of the Grup\-po Na\-zio\-na\-le per l'Anali\-si Ma\-te\-ma\-ti\-ca, la Pro\-ba\-bi\-li\-t\`{a} e le lo\-ro Appli\-ca\-zio\-ni (GNAMPA) of the Isti\-tu\-to Na\-zio\-na\-le di Al\-ta Ma\-te\-ma\-ti\-ca (INdAM). 
The authors are partially supported by FAPESP, project n.~2024/21356-2 (Brazil). 
\\
\textbf{Preprint -- December 2025}}



\date{}


\begin{abstract}
In this paper we propose some continuation theorems for the periodic problem
\begin{equation*}
\begin{cases}
\, x_{i}' = g_{i}(t,x_{i+1}), &i=1,\ldots,n-1,
\\
\, x_{n}' = h(t,x_{1},\ldots,x_{n}),
\\
\, x_{i}(0)=x_{i}(T), &i=1,\ldots,n,
\end{cases}
\end{equation*}
providing a unified framework that improves and extends earlier contributions by Jean Mawhin and collaborators to second-order differential problems governed by nonlinear time-dependent differential operators of the form
\begin{equation*}
\begin{cases}
\, (\phi(t,x'))'=f(t,x,x'),
\\
\, x(0)=x(T),\quad x'(0)=x'(T).
\end{cases}
\end{equation*}
The proof is based on the topological degree theory.
\end{abstract}

\maketitle

\section{Introduction}\label{section-1}

The search for periodic solutions is a central topic in the areas of differential equations and dynamical systems, with wide-ranging applications in physics, biology, and engineering. Among the various techniques developed to establish the existence of such solutions, continuation methods have become a standard tool due to their broad applicability. In this paper, we present new continuation theorems that apply to a wide and general class of periodic problems, providing a unified framework that extends earlier contributions.

The continuation technique typically involves two topological spaces $X$ and $Y$, a continuous map $f \colon X \to Y$, and an equation of the form
\begin{equation}\label{eq-intro-cont}
f(x)=y, \quad x\in D,
\end{equation} 
where $y$ is a given element of $Y$ and $D\subseteq X$.
The key principle of a continuation method is embedding \eqref{eq-intro-cont} in a family of parameter-dependent equations
\begin{equation*}
\mathcal{F}(x,\lambda)=y, \quad \lambda\in\mathopen{[}0,1\mathclose{]},
\end{equation*}
where the function $\mathcal{F}\colon X\times\mathopen{[}0,1\mathclose{]} \to Y$ is such that:
\begin{itemize}
\item $\mathcal{F}(x,1)=f(x)$, for every $x\in X$;
\item the equation $\mathcal{F}(x,0)=y$ has at least one solution in $D$;
\item at least one of the solutions of $\mathcal{F}(x,0)=y$ can be ``continued'' to a solution of $\mathcal{F}(x,1)=y$ in $D$, when $\lambda$ moves from $0$ to $1$, thus yielding a solution of \eqref{eq-intro-cont}.
\end{itemize}
The last property is particularly delicate and the most challenging to establish. Let us consider two simple illustrative examples with $X=\mathbb{R}$ and $D=\mathopen{]}-1,1\mathclose{[}$.
The parameter-dependent equation
\begin{equation*}
(1-\lambda)x+\lambda=0, \quad \lambda\in\mathopen{[}0,1\mathclose{]},
\end{equation*}
admits a unique solution at $x=0$, for $\lambda =0$. 
As $\lambda$ increases, this solution evolves continuously, shifting leftward. 
When $\lambda=1/2$, it exits the domain $D$ through the boundary point  $x=-1$. Hence, the continuation cannot be realized.
Turning to the equation
\begin{equation*}
4x^{2}+2\lambda -1 = 0, \quad \lambda\in\mathopen{[}0,1\mathclose{]},
\end{equation*}
we observe that it has two solutions $x=\pm 1/2$ when $\lambda=0$. As $\lambda$ increases, the two solutions gradually approach each other, merge into the unique solution $x=0$ for $\lambda=1/2$, and then disappear. Also in this situation, the continuation cannot be realized.

From these simple examples it is evident that a way to fulfill the third property is to require that the set of solutions $\{x\in D \colon \mathcal{F}(x,\lambda)=y, \text{for some $\lambda$}\}$ remains away from the boundary $\partial D$ and has a certain property of ``robustness''.
This idea led to the development of the topological degree, initially explored by Leopold Kronecker, Jules Henri Poincar\'{e}, and Luitzen E. J. Brouwer for continuous functions between finite-dimensional vector spaces. Thanks to Jean Leray and Juliusz Pawe\l{} Schauder this notion was then extended to a significant class of mappings defined in Banach spaces, which may have infinite dimension.
We refer the reader to \cite{De-85, DiMa-21, FoGa-95, Ll-78}.

To contextualize our new contributions to the continuation theory, it is first essential to briefly review the main preceding developments.
In 1934, Leray and Schauder published their seminal paper \cite{LeSc-34}, which marked a significant milestone in the fields of nonlinear functional analysis and nonlinear differential equations. 
Their main result, commonly referred to today as the \textit{Leray--Schauder continuation theorem}, states the following (see \cite[Section~14.6]{Ze-86} or \cite[Chapter~2]{DiMa-21} for the proof).
Throughout this paper, we use ``$\mathrm{deg}_\mathrm{LS}$'' to denote the Leray--Schauder degree.

\begin{theorem}[Leray--Schauder continuation theorem]\label{LS-cont-th}
Let $X$ be a Banach space and $\Omega \subseteq X\times \mathopen{[}0,1\mathclose{]}$ be a bounded set which is open in the topology of $X\times \mathopen{[}0,1\mathclose{]}$. Let $\mathcal{G}\colon \Omega \to X$ be a completely continuous function (i.e., $\mathcal{G}$ is continuous and maps bounded subsets into relatively compact subsets).
If
\begin{itemize}
\item $\Sigma\cap\partial\Omega =\emptyset$, where $\Sigma \coloneqq \{(x,\lambda)\in \overline{\Omega} \colon x-\mathcal{G}(x,\lambda)=0 \}$;
\item $\mathrm{deg}_{\mathrm{LS}}(\mathrm{Id}_{X}-\mathcal{G}(\cdot,0),\Omega,0)\neq0$;
\end{itemize}
then there exists a compact connected (continuum) set $C$ in $\Sigma$ along which $\lambda$ takes all values in $\mathopen{[}0,1\mathclose{]}$.
\end{theorem}

We observe that the degree condition ensures that 
\begin{equation*}
S_0 \coloneqq \bigl\{x \in X \colon x-\mathcal{G}(x,0)=0\bigr\} \neq \emptyset.
\end{equation*}
Moreover, we stress that the theorem implies that the subset $C$ of $\Sigma$ connects $S_0$ to $S_1 \coloneqq\{x \in X \colon x-\mathcal{G}(x,1)=0\}$, showing in particular that the equation $x-\mathcal{G}(x,1)=0$ has a solution. 
For additional details, comments, and historical perspectives, the reader is referred to the comprehensive surveys \cite{Ma-97,Ma-99-lux,Ma-99,Ma-05}.

The identification of an open and bounded set $\Omega$ satisfying the first condition in Theorem~\ref{LS-cont-th} is typically achieved by establishing, when possible, a significantly stronger a priori bound condition on the solution set $\Sigma$. Accordingly, many applications of the Leray--Schauder continuation theorem have been developed specifically to address this scenario, and this perspective is central also in the contribution by Jean Mawhin for the study of the periodic problem
\begin{equation}\label{eq-intro-F}
\begin{cases}
\, x' = F(t,x),
\\
\, x(0) = x(T),
\end{cases}
\end{equation}
where $F \colon \mathopen{[}0,T\mathclose{]} \times \mathbb{R}^{m} \to \mathbb{R}^{m}$ is a Carath\'{e}odory vector field, that is,
\begin{itemize}[leftmargin=16pt,labelsep=6pt]
\item for almost every $t\in \mathopen{[}0,T\mathclose{]}$, $F(t,\cdot)$ is continuous;
\item for every $x \in \mathbb{R}^{m}$, $F(\cdot, x)$ is measurable;
\item for every $\rho>0$, there exists a map $\ell\in L^1(\mathopen{[}0,T\mathclose{]},[0,+\infty))$ such that, for almost every $t\in \mathopen{[}0,T\mathclose{]}$ and every $x \in \mathbb{R}^{m}$, with $\Vert x\Vert\leq \rho$, it holds that $\| F(t,x)\|\leq \ell(t)$.
\end{itemize}
Let $X=\mathcal{C}_{T}$ be the Banach space of continuous functions $x\colon \mathopen{[}0,T\mathclose{]}\to\mathbb{R}^{m}$ satisfying $x(0)=x(T)$.
Mawhin's result reads as follows (cf.~\cite[Th\'{e}or\`{e}me~2]{Ma-69} or \cite[Theorem~4.1]{Ma-93}). 
We henceforth denote the Brouwer degree by ``$\mathrm{deg}_{\mathrm{B}}$''.

\begin{theorem}[Mawhin's continuation theorem]\label{th-M-ct}
Let $\Omega\subseteq \mathcal{C}_{T}$ be an open bounded set and suppose that
\begin{itemize}
\item
for each $\lambda\in \mathopen{]}0,1\mathclose{[}$ there is no solution of
\begin{equation*}
\begin{cases}
\, x' = \lambda F(t,x),
\\
\, x(0) = x(T),
\end{cases}
\end{equation*}
with $x\in \partial\Omega$;
\item
the map $F^{\#} \colon \mathbb{R}^{m} \to \mathbb{R}^{m}$, $F^{\#}(w) \coloneqq \frac{1}{T}\int_{0}^{T} F(t,w) \,\mathrm{d}t$ has no zeros on $\partial\Omega \cap \mathbb{R}^{m}$ and
$\mathrm{deg}_{\mathrm{B}}(F^{\#},\Omega\cap\mathbb{R}^{m},0)\neq0$.
\end{itemize}
Then, problem \eqref{eq-intro-F} has a solution in $\overline{\Omega}$.
\end{theorem}

We point out that Mawhin's continuation theorem has the same structure of Theorem~\ref{LS-cont-th}: the first condition requires that no solution lies on the boundary of $\Omega$ along the homotopy, while the second condition involves the topological degree. The novelty of Theorem~\ref{th-M-ct} is that the finite-dimensional Brouwer degree can be employed, thanks to the application of a reduction property when $\lambda=0$.

From now on, given our interest in solutions to the original problem \eqref{eq-intro-F}, we refrain from explicitly mentioning the existence of the continuum of solutions for $\lambda\in\mathopen{[}0,1\mathclose{]}$.

In the context of periodic problems, it is common to deal with perturbations of autonomous systems and thus encounter equations of the form
\begin{equation*}
x' = F_{0}(x) + P(t,x),
\end{equation*}
where $F_{0}$ is an autonomous vector field. In this framework, it is natural to introduce the parameter-dependent equation
\begin{equation*}
x' = F_{0}(x) + \lambda P(t,x), \quad \lambda\in\mathopen{[}0,1\mathclose{]},
\end{equation*}
with the aim of applying a continuation principle to reduce the study to the autonomous periodic problem
\begin{equation*}
\begin{cases}
\, x' = F_{0}(x),
\\
\, x(0) = x(T).
\end{cases}
\end{equation*}
With this perspective in mind, in \cite{CMZ-92} Anna Capietto, Mawhin, and Fabio Zanolin proposed a second continuation theorem (see also \cite{BaMa-91,Ma-23}), where the parameter-dependent family of problems is given by a Carath\'{e}odory vector field
\begin{equation*}
\tilde{F}=\tilde{F}(t,x,\lambda)\colon \mathopen{[}0,T\mathclose{]}\times\mathbb{R}^{m}\times \mathopen{[}0,1\mathclose{]} \to \mathbb{R}^{m}
\end{equation*}
such that
\begin{equation*}
\tilde{F}(t,x,0) = F_{0}(x), \qquad \tilde{F}(t,x,1) = F(t,x).
\end{equation*}
Notice that the equation is autonomous for $\lambda=0$. The result is the following (cf.~\cite[Theorem~2]{CMZ-92}).

\begin{theorem}[Capietto--Mawhin--Zanolin continuation theorem]\label{th-CMZ-ct}
Let $\Omega\subseteq \mathcal{C}_{T}$ be an open bounded set and suppose that
\begin{itemize}
\item
for each $\lambda\in\mathopen{[}0,1\mathclose{[}$ there is no solution of
\begin{equation*}
\begin{cases}
\, x' = \tilde{F}(t,x,\lambda),
\\
\, x(0) = x(T),
\end{cases}
\end{equation*}
with $x\in \partial\Omega$;
\item
$\mathrm{deg}_{\mathrm{B}}(F_{0},\Omega\cap\mathbb{R}^{m},0)\not=0$.
\end{itemize}
Then, problem \eqref{eq-intro-F} has a solution in $\overline{\Omega}$.
\end{theorem}

Theorem~\ref{th-M-ct} and Theorem~\ref{th-CMZ-ct} admit extensions that allow to study the existence of periodic solutions of second-order differential equations of the form
\begin{equation*}
x'' = f(t,x,x'),
\end{equation*}
where $f \colon \mathopen{[}0,T\mathclose{]} \times \mathbb{R}^{m} \times \mathbb{R}^{m} \to \mathbb{R}^{m}$ is a Carath\'{e}odory vector field.
As outlined above, the continuation method involves embedding the above equation into a parameter-dependent family 
\begin{equation*}
x'' = \tilde{f}(t,x,x',\lambda), \quad \lambda\in\mathopen{[}0,1\mathclose{]},
\end{equation*}
where in the case of Theorem~\ref{th-M-ct} $\tilde{f}$ is of the form
\begin{equation}\label{eq-homotopy-M-ct}
\tilde{f}(t,x,x',\lambda) = \lambda f(t,x,x'),
\end{equation}
and in the case of Theorem~\ref{th-CMZ-ct} the map $\tilde{f}$ satisfies
\begin{equation}\label{eq-homotopy-CMZ-ct}
\tilde{f}(t,x,x',0) = f_{0}(x,x'), \qquad \tilde{f}(t,x,x',1) = f(t,x,x'),
\end{equation}
for some autonomous function $f_{0}$.
The corresponding statements are straightforward adaptations of Theorem~\ref{th-M-ct} and Theorem~\ref{th-CMZ-ct}; for brevity, we omit their explicit formulations and refer the reader to \cite[Th\'{e}or\`{e}me~2]{Ma-69}, \cite[Corollary~6]{CMZ-92}, and \cite[Chapter~6]{Ma-95}.

In recent decades, the interest in nonlinear differential operators, such as the $p$-Laplacian and curvature operators, has grown significantly, largely driven by their wide range of applications; consequently, new continuation theorems have emerged within the framework of topological methods. 
A crucial milestone in this direction was established by Ra\'{u}l Man\'{a}sevich and Mawhin in \cite{MaMa-98}, as reflected by its extensive citation record.
Their investigation concerns the periodic solutions of the second-order differential equation
\begin{equation*}
(\phi(x'))' = f(t,x,x'),
\end{equation*}
where $\phi \colon \mathbb{R}^{m} \to \mathbb{R}^{m}$ is a continuous function satisfying
\begin{itemize}[leftmargin=32pt,labelsep=10pt]
\item[$(\textsc{h}_{1})$] $\langle\phi(x_{1})-\phi(x_{2}),x_{1}-x_{2}\rangle>0$, for every $x_{1},x_{2}\in \mathbb{R}^{m}$ with $x_{1}\neq x_{2}$;
\item[$(\textsc{h}_{2})$] there exists a function $\alpha \colon \mathopen{[}0,+\infty\mathclose{[} \to \mathopen{[}0,+\infty\mathclose{[}$, with $\alpha(s)\to+\infty$ as $s\to+\infty$,
such that $\langle\phi(x),x\rangle\geq\alpha(\|x\|)\|x\|$, for all $x\in\mathbb{R}^{m}$.
\end{itemize}
Their continuation results are applied to the homotopic equation
\begin{equation}\label{eq-phi-tilde}
(\phi(x'))' + \tilde{f}(t,x,x',\lambda) = 0, \quad \lambda\in\mathopen{[}0,1\mathclose{]},
\end{equation}
with $\tilde{f}$ as in \eqref{eq-homotopy-M-ct} in their version of Theorem~\ref{th-M-ct} (cf.~\cite[Theorem~3.1]{MaMa-98}), while $\tilde{f}$ satisfies \eqref{eq-homotopy-CMZ-ct} in their version of Theorem~\ref{th-CMZ-ct} (cf.~\cite[Theorem~4.1]{MaMa-98}).

The technical conditions $(\textsc{h}_{1})$ and $(\textsc{h}_{2})$ are exploited in \cite{MaMa-98} (and in several subsequent papers) to construct a completely continuous operator whose fixed points correspond to periodic solutions of \eqref{eq-phi-tilde}.
It is easy to prove that from these two conditions it follows that the map $\phi \colon \mathbb{R}^{m} \to \mathbb{R}^{m}$ is a homeomorphism such that $\phi(0)=0$ (cf.~\cite[\S~11]{De-85}).
In~\cite{FeZa-17} Guglielmo Feltrin and Zanolin showed that these technical assumptions can be removed, requiring only that $\phi$ is a homeomorphism with $\phi(0)=0$. Then, \cite{FeZa-17} generalizes the results in~\cite{MaMa-98}, also in the context of cyclic-type systems.

In recent decades, attention has also grown toward $p$-Laplacian-type operators with variable exponents, specifically those of the form  
$\mathrm{div}\big(|\nabla u|^{p(x)-2} \nabla u\big)$. These operators represent an important generalization of the classical $p$-Laplacian, offering greater flexibility in modeling complex nonlinear phenomena in applied sciences, such as fluid dynamics and the behaviour of porous media. In such contexts, some materials cannot be accurately modeled using the classical Lebesgue and Sobolev spaces \(L^p\) and \(W^{1,p}\), and it becomes essential to allow the exponent to vary. This need naturally arises in various applications, including phase transitions and damage mechanics. For a foundational treatment of variable exponent spaces, see \cite{Or-31}; for further examples and insights, we refer to \cite{RaRe-15}.

In this perspective, in \cite{GHMMT-24}, Marta Garc\'{i}a-Huidobro, Man\'{a}sevich, Mawhin, and Satoshi Tanaka generalized Theorem~\ref{th-M-ct} (and also \cite[Theorem~3.1]{MaMa-98}) for the $T$-periodic problem associated with
\begin{equation}\label{eq-GHMMT-24}
(\phi(t,x'))'=f(t,x,x'),
\end{equation}
where $\phi \colon \mathopen{[}0,T\mathclose{]}\times\mathbb{R}^{m}\to \mathbb{R}^{m}$ is a continuous function satisfying
\begin{itemize}[leftmargin=32pt,labelsep=10pt]
\item $\phi(0,x)=\phi(T,x)$, for every $x\in\mathbb{R}^{m}$;
\item for every $t\in\mathopen{[}0,T\mathclose{]}$, $\phi(t,x)=0$ if and only if $x=0$;
\end{itemize}
together with
\begin{itemize}[leftmargin=32pt,labelsep=10pt]
\item[$(\textsc{h}_{1}')$] $\langle \phi(t,x_{1})-\phi(t,x_{2}), x_{1}-x_{2} \rangle >0$, for every $t\in\mathopen{[}0,T\mathclose{]}$ and $x_{1},x_{2}\in\mathbb{R}^{m}$ with $x_{1}\neq x_{2}$;
\item[$(\textsc{h}_2')$] there exists a function $\alpha \colon \mathopen{[}0,+\infty\mathclose{[} \to \mathopen{[}0,+\infty\mathclose{[}$, with $\alpha(s)\to+\infty$ as $s\to+\infty$,
such that $\langle\phi(t,x), x\rangle\geq\alpha(\|x\|)\|x\|$, for every $t\in\mathopen{[}0,T\mathclose{]}$ and $x\in\mathbb{R}^{m}$.
\end{itemize}
We refer to \cite[Theorem~4.1]{GHMMT-24} for the precise statement. 
We note, however, that \cite{GHMMT-24} does not treat the case when the right-hand side of \eqref{eq-GHMMT-24} depends nonlinearly on $\lambda$ (cf.~Theorem~\ref{th-CMZ-ct}), and, to the best of our knowledge, it remains unexplored. We refer also to \cite{GHMMT-25,YSLC-25} for related result in this context.

Having reviewed the main foundational works of this line of research, we are now in a position to present our contribution. In the spirit of \cite{FeZa-17}, we investigate the following periodic problem
\begin{equation}\label{main-syst-intro}
\begin{cases}
\, x_{i}' = g_{i}(t,x_{i+1}), &i=1,\ldots,n-1,
\\
\, x_{n}' = h(t,x_{1},\ldots,x_{n}),
\\
\, x_{i}(0)=x_{i}(T), &i=1,\ldots,n,
\end{cases}
\end{equation}
which is a system of $n$ vector differential equations with periodic boundary conditions, being $T>0$ fixed.
This family of problems includes as special cases all the aforementioned second-order equations, possibly involving nonlinear differential operators that depend explicitly on time. In addition, it also covers higher-order vector differential equations. For example, a $n$-th order differential equation of the form
\begin{equation*}
(\phi(t,x'))^{(n-1)} = h(t,x,x',\ldots,x^{(n-1)})
\end{equation*}
can be reformulated as a first-order system of $n$ equations
\begin{equation*}
\begin{cases}
\, x_{1}'=\psi(t,x_{2}),
\\
\, x_{j}'=x_{j+1}, &j=2,\ldots,n-1,
\\
\, x_{n}'= h(t,x_{1},x_{2},\ldots,x_{n-1}),
\end{cases}
\end{equation*}
where $\psi(t,\cdot)$ denotes the inverse of $\phi(t,\cdot)$; in particular, the first equation corresponds to the equality $x_{2}(t)=\phi(t,x_{1}'(t))$ for a.e.~$t\in\mathopen{[}0,T\mathclose{]}$.

As will be discussed later (see Remark~\ref{rem-4.1}), our approach allows to remove the unnecessary technical conditions $(\textsc{h}_{1}')$ and $(\textsc{h}_{2}')$ assumed in \cite{GHMMT-24} and, moreover, to present a more general version of Theorem~\ref{th-CMZ-ct} in the context of \eqref{main-syst-intro}.

Furthermore, in order to include as well differential operators such as the mean curvature and the relativistic Minkowski ones, namely,
\begin{equation*}
\phi(s) = \frac{s}{\sqrt{1 + |s|^2}}, \qquad \phi(s) = \frac{s}{\sqrt{1 - |s|^2}},
\end{equation*}
we will consider continuous functions $\phi \colon \mathopen{[}0,T\mathclose{]}\times U \to V$, with $U$ and $V$ open subsets of $\mathbb{R}^{m}$ containing $0$, such that, for every $t\in \mathopen{[}0,T\mathclose{]}$, $\phi(t,\cdot)$ is a homeomorphism with $\phi(t,0)=0$, and $\phi(0,s)=\phi(T,s)$ for every $s\in U$.
To the best of our knowledge, this is the first result that applies to singular (i.e., $U\neq\mathbb{R}^{m}$) and bounded (i.e., $V\subseteq\mathbb{R}^{m}$ bounded) differential operators $\phi$ depending on time.

Our approach relies on the coincidence degree theory developed by Mawhin (cf.~\cite{GaMa-77,Ma-79}), together with the reduction formula established in \cite{FeZa-17}.

The main goal of our paper is to formulate two new continuation theorems. For this reason, we do not focus on specific applications to second-order (or higher-order) problems here, pointing out that the results established in earlier works under similar assumptions on the right-hand side of \eqref{eq-GHMMT-24} or \eqref{main-syst-intro} continue to be valid within our more general framework. For the same reason, we also do not discuss the broad literature in which these methods have been applied.

\smallskip

The plan of the paper is as follows. In Section~\ref{section-2}, we recall the reduction formula developed in \cite{FeZa-17}, which is a key ingredient for the proof of our main results, which are two continuation theorems for \eqref{main-syst-intro} presented in Section~\ref{section-3}.
Next, in Section~\ref{section-4}, we explore the implications for second-order vector differential equations, providing improvements to all the previously established results discussed in this introduction.

\section{A degree reduction formula}\label{section-2}

In this section, we review a result discussed in \cite[Section~2]{FeZa-17}.
Specifically, the goal is to present a reduction formula for the computation of the coincidence degree associated with an equation of the form
\begin{equation*}
Lx = N(x),\quad x\in \mathrm{dom}\,L\cap \mathrm{dom}\,N,
\end{equation*}
where $L$ and $N$ are defined in the cartesian product of normed spaces.

For $i=1,\ldots,n$, let $(X_{i},\|\cdot\|_{X_{i}})$ and $(Z_{i},\|\cdot\|_{Z_{i}})$ be real normed linear spaces and consider
\begin{equation*}
X\coloneqq \prod_{i=1}^{n} X_{i}, \qquad Z\coloneqq  \prod_{i=1}^{n} Z_{i},
\end{equation*}
endowed with the norms $\|\cdot\|_{X}=\sum_{i} \|\cdot\|_{X_{i}}$ and $\|\cdot\|_{Z}=\sum_{i} \|\cdot\|_{Z_{i}}$, respectively.

For $i=1,\ldots,n$, let
\begin{equation*}
L_{i} \colon \mathrm{dom}\,L_{i} (\subseteq X_{i}) \to Z_{i}
\end{equation*}
be a not necessarily bounded Fredholm linear operator of index zero. We recall that Fredholm means that $\ker L_{i}$ has finite dimension, $\mathrm{Im}\,L_{i}$ is closed in $Z_{i}$ and has finite codimension; the index of $L_i$ is defined as $\dim(\ker L_{i})-\mathrm{codim}(\mathrm{Im}\,L_i)$.

Since $L_{i}$ is Fredholm, there exist two linear continuous projections
\begin{equation*}
P_{i} \colon X_{i} \to X_{i}, \qquad Q_{i} \colon Z_{i} \to Z_{i},
\end{equation*}
such that
\begin{equation*}
\mathrm{Im}\, P_i=\ker L_{i}, \qquad \ker Q_i = \mathrm{Im}\, L_i.
\end{equation*}
This implies that 
\begin{equation*}
X_{i} = \ker L_{i} \oplus \ker P_{i}, \qquad Z_{i} = \mathrm{Im}\,L_{i} \oplus \mathrm{Im}\,Q_{i}.
\end{equation*}

Let $K_{i} \colon \mathrm{Im}\,L_{i} \to \mathrm{dom}\,L_{i} \cap \ker P_{i}$ be the right inverse of $L_{i}$ (depending on $P_i$), i.e., the linear map such that $L_{i} K_{i}v= v$ for each $v\in \mathrm{Im}\,L_{i}$.

We fix an orientation on the finite-dimensional spaces $\ker L_{i}$ and $\mathrm{Im}\, Q_{i}$, and a linear orientation-preserving isomorphism $J_{i} \colon \mathrm{Im}\, Q_{i} \to \ker L_{i}$.

Let $\mathrm{dom}\,L\coloneqq \prod_{i=1}^{n} \mathrm{dom}\,L_{i}\subseteq X$ and $L \colon \mathrm{dom}\,L \to Z$ be defined as
\begin{equation*}
Lx\coloneqq (L_{1}x_{1},\ldots,L_{n}x_{n}), \quad x=(x_{1},\ldots,x_{n})\in \mathrm{dom}\,L.
\end{equation*}
As a consequence, we have
\begin{equation*}
\ker L = \prod_{i=1}^{n} \ker L_{i} \subseteq X, \qquad \mathrm{Im}\,L = \prod_{i=1}^{n} \mathrm{Im}\,L_{i} \subseteq Z,
\end{equation*}
and hence $L$ is a Fredholm linear mapping of index zero.
The map $K \colon \mathrm{Im}\,L \to \mathrm{dom}\,L \cap \prod_{i=1}^{n}\ker P_{i}$, defined as
\begin{equation*}
Kz\coloneqq (K_{1}z_{1},\ldots,K_{n}z_{n}), \quad z=(z_{1},\ldots,z_{n})\in \mathrm{Im}\,L, 
\end{equation*}
is the right inverse of $L$.

Moreover, we define 
\begin{align*}
&P \colon X \to X, 
&&Px\coloneqq (P_{1}x_{1},\ldots,P_{n}x_{n}), 
&& x=(x_{1},\ldots,x_{n})\in X,
\\
&Q \colon Z \to Z, 
&&Qz\coloneqq (Q_{1}z_{1},\ldots,Q_{n}z_{n}), 
&& z=(z_{1},\ldots,z_{n})\in Z,
\\
&J \colon \mathrm{Im}\, Q \to \ker L, 
&&Jz\coloneqq (J_{1}z_{1},\ldots,J_{n}z_{n}), 
&& z=(z_{1},\ldots,z_{n})\in \mathrm{Im}\, Q.
\end{align*}

Consider now an $L$-completely continuous map
\begin{equation*}
N =(N_{1},\ldots,N_{n}) \colon \mathrm{dom}\,N (\subseteq X) \to Z,
\end{equation*}
where $L$-completely continuous means that $N$ and $K(\mathrm{Id}_{Z}-Q)N$ are continuous, and, for any bounded subset $B$ of $\mathrm{dom}\, N$, the sets $QN(B)$ and $K(\mathrm{Id}_{Z}-Q)N(B)$ are relatively compact in $Z$ and $X$, respectively.

With the previous assumptions on $L$ and $N$, we study the coincidence equation
\begin{equation}
\label{equationLN}
Lx = N(x),\quad x\in \mathrm{dom}\,L\cap \mathrm{dom}\,N.
\end{equation}
Consider the nonlinear operator $\Phi \colon \mathrm{dom}\,N\to X$, defined by
\begin{equation*}
\Phi(x)\coloneqq Px + JQN(x) + K(\mathrm{Id}_{Z}-Q)N(x), \quad x\in \mathrm{dom}\,N.
\end{equation*}
Observe that the $L$-complete continuity of $N$ implies the complete continuity of $\Phi$.
It is not difficult to prove that equation \eqref{equationLN} is 
equivalent to the fixed point problem
\begin{equation*}
x = \Phi(x), \quad x\in \mathrm{dom}\,N.
\end{equation*}
Consequently, in this framework the classical topological degree theory applies. We recall that the coincidence degree $\mathrm{D}_{L}(L-N,\Omega)$ of $L$ and $N$ in an open (possibly unbounded) set $\Omega$ is well-defined if the set $\{x\in\Omega \colon Lx=N(x)\}$ is compact. In this case, we define
\begin{equation*}
\mathrm{D}_{L}(L-N,\Omega)=\mathrm{deg}_{\mathrm{LS}}(\mathrm{Id}_{X}-\Phi,\Omega,0)
\end{equation*}
and $\mathrm{D}_{L}$ inherits all the classical properties of the topological degree (for more details, we refer to \cite{GaMa-77,Ma-79}).

From now on, we assume that the $n$ components $N_i \colon \mathrm{dom}\,N \to Z_i$ of $N$ have the following form:
\begin{equation*}
\begin{cases}
\, N_{i}(x_{i+1}), & i=1,\ldots,n-1,
\\
\, N_{n}(x_{1},\ldots,x_{n}),
\end{cases}
\qquad x=(x_{1},\ldots,x_{n})\in \mathrm{dom}\,N.
\end{equation*}

Given an open set $\Omega\subseteq X$, for $i=1,\ldots,n$, let
\begin{equation*}
\Omega_{i} \coloneqq \bigl{\{} w\in X_{i} \colon \exists \, (x_{1}, \ldots,x_{n})\in\Omega \text{ such that } x_{i}=w\bigr{\}};
\end{equation*}
in other words, $\Omega_{i}$ is the projection of $\Omega$ on $X_{i}$. 
From now on and in the rest of this section, we consider sets $\Omega$ contained in $\mathrm{dom}\,N$.
For $i=1,\ldots,n-1$, we define the maps
\begin{equation}\label{def-eta_i}
\eta_{i}\colon\Omega_{i+1}\cap \ker L_{i+1} \to \ker L_{i}, \qquad \eta_{i}(w) \coloneqq -J_{i}Q_{i}N_{i} (w), 
\end{equation}

We can now state a reduction formula for the coincidence degree of $L$ and $N$ in an open set $\Omega$.

\begin{proposition}[Reduction formula]\label{th-2.6-FeltrinZanolin}
Let $L$ and $N$ be as above. 
Let $\Omega$ be an open (possibly unbounded) set in $X$ with $\Omega\subseteq \mathrm{dom}\,N$.
For $\vartheta\in\mathopen{]}0,1\mathclose{]}$, consider the following system
\begin{equation}\label{pb-P-theta}
\begin{cases}
\, L_{i}x_{i} = N_{i}(x_{i+1}), & i=1,\ldots,n-1, \\
\, L_{n}x_{n} = \vartheta  N_{n}(x_{1},\ldots,x_{n}),
\end{cases}
\end{equation}
and define
\begin{equation*}
\mathcal{S}_{\vartheta} \coloneqq
\bigl{\{}x\in\Omega\cap \mathrm{dom}\,L \colon x \text{ is a solution of \eqref{pb-P-theta}}\bigr{\}},
\quad \vartheta\in\mathopen{]}0,1\mathclose{]}.
\end{equation*}
Assume that
\begin{itemize}
\item [$(a_{1})$] $\mathrm{Im}\,L_{i}\cap N_{i}(\Omega_{i+1} \cap \ker L_{i+1}) \subseteq \{0_{Z_{i}}\}$, for all $i=1,\ldots,n-1$;
\item [$(a_{2})$] there exists a compact set $\mathcal{K}\subseteq\Omega$ such that $\mathcal{S}_{\vartheta}\subseteq\mathcal{K}$, for all $\vartheta\in\mathopen{]}0,1\mathclose{]}$;
\item [$(a_{3})$] the set $\mathcal{S}_{0} \coloneqq \{x\in\Omega\cap \ker L \colon QN(x)=0\}$ is compact.
\end{itemize}
Then,
\begin{equation*}
\mathrm{D}_{L}(L-N,\Omega) = \mathrm{deg}_{\mathrm{B}}(\mathcal{N},\Omega \cap \ker L,0),
\end{equation*}
where $\mathcal{N}\colon \Omega\cap \ker L \to \ker L$ is defined as
\begin{equation}\label{def-mathcal-N}
\mathcal{N}(x)\coloneqq \bigl{(}\eta_{1}(x_{2}),\ldots,\eta_{n-1}(x_{n}),-J_{n}Q_{n}N_{n}|_{\Omega\cap\ker L} (x) \bigr{)},
\quad x=(x_{1},\ldots,x_{n}).
\end{equation}
\end{proposition}

\begin{proof}[Sketch of the proof]
The proof is contained in \cite[Section~2]{FeZa-17} and it relies, inter alia, on the homotopy invariance property of the coincidence degree.
For the reader's convenience, we now summarize the main steps of the proof. We refer to \cite[Theorem~2.6]{FeZa-17} for more details.

\smallskip

\noindent
\textit{Step~1.} For $\vartheta\in[0,1]$, we consider the system
\begin{equation}\label{auxiliary-system}
\begin{cases}
\, L_i x_i = N_{i}(x_{i+1}), &  i=1, \dots n-1,
\\
\, L_n x_n= \vartheta N_{n}(x)+(1-\vartheta) Q_nN_{n}(x),
\end{cases}
\end{equation}
and we denote by $\widetilde{N} \colon \Omega \times [0,1] \to Z$ the continuous map whose components are given by the right-hand sides of \eqref{auxiliary-system}.
We point out that 
\begin{equation*}
\bigl{\{}x\in\Omega\cap \mathrm{dom}\,L \colon x \text{ is a solution of \eqref{auxiliary-system}}\bigr{\}}=\mathcal{S}_{\vartheta},
\end{equation*}
for every $\vartheta\in\mathopen{]}0,1\mathclose{]}$.
Therefore, using $(a_{2})$ and $(a_{3})$, thanks to the homotopy invariance property of the coincidence degree we have
\begin{equation*}
\mathrm{D}_{L}(L-\widetilde N(\cdot,1),\Omega)
=\mathrm{D}_{L}(L-\widetilde N(\cdot,0),\Omega).
\end{equation*}
Hence,
\begin{equation*}
\mathrm{D}_{L}(L- N,\Omega) = \mathrm{D}_{L}(L-\widetilde N(\cdot,0),\Omega).
\end{equation*}

\smallskip

\noindent
\textit{Step~2.} Consider the operator $\widehat{N} \colon \Omega  \to Z$, given by $\widehat N = \widetilde N(\cdot,0)$, that is,
\begin{equation*}
\begin{cases}
\, \widehat N_i (x_{i+1})= N_{i}(x_{i+1}), &i=1, \dots n-1,
\\
\, \widehat N_n (x)= Q_nN_{n}(x),
\end{cases}
\end{equation*}
and the homotopy map $\widehat \Phi \colon \mathopen{[}0,1\mathclose{]} \times \Omega \to X$ defined as
\begin{equation*}
\widehat \Phi(\mu,x)\coloneqq Px + JQ\widehat N(x) + \mu K(\mathrm{Id}_{Z}-Q)\widehat N(x).
\end{equation*}
Following \cite[Lemma~2.4]{FeZa-17} and exploiting hypothesis $(a_{1})$, we deduce that the set $\{(\mu,x)\in\mathopen{[}0,1\mathclose{]}\times\Omega \colon x=\widehat{\Phi}(\mu,x)\}$ is compact.
By the homotopy invariance property, we have
\begin{equation*}
\mathrm{D}_{L}(L-\widehat{N},\Omega)
=
\mathrm{deg}_{\mathrm{LS}}(\mathrm{Id}_X-\widehat \Phi(1,\cdot),\Omega,0)
=
\mathrm{deg}_{\mathrm{LS}}(\mathrm{Id}_X-\widehat \Phi(0,\cdot),\Omega,0),
\end{equation*}
where the first equality is due to the definition of the coincidence degree.

\smallskip

\noindent
\textit{Step~3.} We notice first that $\widehat \Phi(0,\Omega)\subseteq \ker L$. An application of the classical finite-dimensional reduction property of the Leray--Schauder degree (cf.~\cite[Theorem~8.7]{De-85}) leads to
\begin{equation*}
\mathrm{deg}_{\mathrm{LS}}(\mathrm{Id}_X-\widehat \Phi(0,\cdot),\Omega,0)
=
\mathrm{deg}_{\mathrm{B}}(\mathcal{N}, \Omega\cap \ker L,0),
\end{equation*}
where $\mathcal{N}$ is as in formula \eqref{def-mathcal-N}, and thus the thesis follows. 
\end{proof}

The next result deals with the computation of $\mathrm{deg}_{\mathrm{B}}(\mathcal{N}, \Omega\cap \ker L,0)$ via an application of the product property of the degree (cf.~\cite[Theorem~11.3]{Br-14}). 
Recalling the definition \eqref{def-eta_i} of the maps $\eta_i$, $i=1,\ldots,n-1$, now we need to define the map
\begin{equation*}
\eta_{n}\colon\widetilde{\Omega}_{1} \to \ker L_{n}, \quad \eta_{n}(w) \coloneqq -J_{n}Q_{n}N_{n}(w,0,\ldots,0), 
\end{equation*}
where $\widetilde{\Omega}_{1}\coloneqq \{w\in \ker L_{1} \colon (w,0,\ldots,0)\in \Omega\}$ is assumed to be nonempty.

\begin{proposition}[Product formula]\label{lem-2.8-FeltrinZanolin}
Let $L$, $N$ and $\Omega$ be as in Proposition~\ref{th-2.6-FeltrinZanolin}. 
Let $\mathcal{N}$ be defined as in \eqref{def-mathcal-N}, assume that the degree $\mathrm{deg}_{\mathrm{B}}(\mathcal{N},\Omega \cap \ker L,0)$ is well-defined
and suppose that the following conditions hold:
\begin{itemize}
\item[$(a_{4})$] $\mathrm{dim}(\ker L_{i}) = d$, for all $i=1,\ldots,n$;
\item[$(a_{5})$] $0_{X_{i}}\in \Omega_{i}$, for all $i=2,\ldots,n$;
\item[$(a_{6})$] $\{w\in\Omega_{i+1}\cap\ker L_{i+1} \colon \eta_{i}(w)=0_{X_{i}}\}=\{0_{X_{i+1}}\}$, for all $i=1,\ldots,n-1$.
\end{itemize}
Then,
\begin{equation*}
\begin{aligned}
&\mathrm{deg}_{\mathrm{B}}(\mathcal{N},\Omega \cap \ker L,0) =
\\ &= (-1)^{d(n+1)} \, \mathrm{deg}_{\mathrm{B}}(\eta_{n},\widetilde{\Omega}_{1},0) \cdot \prod_{i=1}^{n-1} \mathrm{deg}_{\mathrm{B}}(\eta_{i},\Omega_{i+1} \cap \ker L_{i+1},0).
\end{aligned}
\end{equation*}
\end{proposition}

\begin{proof}[Sketch of the proof]
For reader's convenience, we summarize the idea of the proof and refer to \cite[Lemma~2.8]{FeZa-17} for more details.

\smallskip

\noindent
\textit{Step~1.} Let
\begin{equation*}
\widetilde{\Omega} = \widetilde{\Omega}_{1} \times (\Omega_{2}\cap\ker L_{2}) \times \dots \times (\Omega_{n}\cap\ker L_{n})
\end{equation*}
and let $\eta \colon\widetilde{\Omega} \to \ker L$ be defined as
\begin{equation*}
\eta (x) = \bigl{(} \eta_{1}(x_{2}), \ldots, \eta_{n-1}(x_{n}), \eta_{n}(x_{1}) \bigr{)}, \quad x=(x_{1},\ldots,x_{n})\in\widetilde{\Omega}.
\end{equation*}
In order to link $\mathcal{N}$ and $\eta$, we introduce a suitable open subset $W$ of $\Omega\cap\ker L$ and we consider the homotopy $\widetilde{\mathcal{N}} \colon \mathopen{[}0,1\mathclose{]} \times W \to \ker L$ with components
\begin{equation*}
\begin{cases}
\, \widetilde{\mathcal{N}}_{i}(\vartheta,x) = \eta_{i}(x_{i+1}), & i=1,\ldots,n-1,
\\
\, \widetilde{\mathcal{N}}_{n}(\vartheta,x) = - J_n Q_{n}N_{n}(x_{1},\vartheta x_{2},\ldots,\vartheta x_{n}).
\end{cases}
\end{equation*}
To construct $W$ we proceed as follows. First of all, we observe that 
\begin{itemize}
\item if $(x_1,\ldots,x_n)\in\mathcal{S}_\eta\coloneqq\{x\in\widetilde{\Omega}:\eta(x)=0\}$, then $x_2=\dots=x_n=0$ by $(a_{6})$; this implies, as a byproduct, that $\mathcal{S}_\eta\subseteq \Omega$;
\item if $(x_1,\ldots,x_n)\in\mathcal{S}_{\mathcal{N}}\coloneqq\{x\in\Omega\cap \ker L:\mathcal{N}(x)=0\}$, then $x_2=\dots=x_n=0$ as well, still using $(a_{6})$.
\end{itemize}
Consequently, we have $\mathcal{S}_\eta= \mathcal{S}_{\mathcal{N}}$ and denote $\mathcal{S}\coloneqq \mathcal{S}_\eta= \mathcal{S}_{\mathcal{N}}$. 
Next, for every $(w,0,\ldots,0)\in \mathcal{S}$, let $\varepsilon_{w}>0$ be sufficiently small in such a way that the open neighborhood $\mathcal{B}_{w}\coloneqq B(w,\varepsilon_{w})\times B(0,\varepsilon_{w}) \times \ldots \times B(0,\varepsilon_{w})$ of $(w,0,\ldots,0)$ in $\mathbb{R}^{dn}$ is contained in $\Omega$. 
Then, we define 
\begin{equation*}
W\coloneqq \bigcup_{(w,0,\ldots,0)\in \mathcal{S}} \mathcal{B}_{w}.
\end{equation*}

We show that $\widetilde{\mathcal{N}}$, which is not necessarily well defined in $\mathopen{[}0,1\mathclose{]} \times (\Omega\cap\ker L)$, is actually well defined in $\mathopen{[}0,1\mathclose{]} \times W$. Indeed, take $(x_1,x_2,\ldots,x_n)\in W$ and let $(w,0,\ldots,0)\in \mathcal{S}$ be such that $(x_1,x_2,\ldots,x_n)\in \mathcal{B}_{w}$. By the definition of $\mathcal{B}_{w}$ as a product of balls of $\mathbb{R}^d$, which are centered at zero except for the first one, it follows that $(x_1,\vartheta x_2,\ldots,\vartheta x_n)\in \mathcal{B}_{w}$ for every $\vartheta\in [0,1]$. Therefore, the homotopy $\widetilde{\mathcal{N}}$ is correctly defined.

By $(a_{6})$, we have that $\bigcup_{\vartheta\in\mathopen{[}0,1\mathclose{]}}\{x\in W \colon \widetilde{\mathcal{N}}(\vartheta,x)=0\}$ coincides with $[0,1]\times \mathcal{S}$ and is  compact.
Therefore, from the homotopy invariance property we deduce that
\begin{equation*}
\mathrm{deg}_{\mathrm{B}}(\mathcal{N},W,0) 
= \mathrm{deg}_{\mathrm{B}}(\eta,W,0).
\end{equation*}
Then, recalling the definition of $\mathcal{S}$, we can apply the excision property to $\mathcal{N}$ and $\eta$ obtaining
\begin{equation*} 
\mathrm{deg}_{\mathrm{B}}(\mathcal{N},W,0) 
= \mathrm{deg}_{\mathrm{B}}(\mathcal{N},\Omega \cap \ker L,0)
\end{equation*}
and
\begin{equation*}
\mathrm{deg}_{\mathrm{B}}(\eta,W,0)= \mathrm{deg}_{\mathrm{B}}(\eta,\widetilde{\Omega},0).
\end{equation*}
Therefore,
\begin{equation*}
\mathrm{deg}_{\mathrm{B}}(\mathcal{N},\Omega\cap \ker L,0)= \mathrm{deg}_{\mathrm{B}}(\eta,\widetilde{\Omega},0).
\end{equation*}

\smallskip

\noindent
\textit{Step~2.} 
Let $\tilde{\eta} \colon \widetilde{\Omega} \to \ker L$ be defined as
\begin{equation*}
\tilde{\eta} (x) = \bigl{(} \eta_{n}(x_{1}), \eta_{1}(x_{2}), \ldots, \eta_{n-1}(x_{n}) \bigr{)}, \quad x=(x_{1},\ldots,x_{n})\in\widetilde{\Omega}.
\end{equation*}
We notice that $\tilde{\eta}(x) = A \eta(x)$ for all $x\in \widetilde{\Omega}$, where
\begin{equation*}
A =
\begin{pmatrix}
0 & 0 & \cdots & 0 & \mathrm{Id}_{\mathbb{R}^{d}} \\
\mathrm{Id}_{\mathbb{R}^{d}} & 0 & \cdots & 0 & 0 \\
0 & \ddots  & \ddots & \vdots & \vdots \\
\vdots & \ddots & \mathrm{Id}_{\mathbb{R}^{d}} & 0 & 0 \\
0 & \cdots & 0 & \mathrm{Id}_{\mathbb{R}^{d}} & 0
 \end{pmatrix}
\in \mathbb{R}^{dn\times dn}
\end{equation*}
is a permutation matrix with determinant $\det (A)=(-1)^{d(n+1)}$.
Hence, we have 
\begin{equation*}
\mathrm{deg}_{\mathrm{B}}(\tilde{\eta},\widetilde{\Omega},0) 
= \mathrm{sign}(\det (A)) \, \mathrm{deg}_{\mathrm{B}}(\eta,\widetilde{\Omega},0)
= (-1)^{d(n+1)} \, \mathrm{deg}_{\mathrm{B}}(\eta,\widetilde{\Omega},0)
\end{equation*}
(cf.~\cite[Lemma~1.3.1]{DiMa-21} or \cite[Theorem~2.10]{FoGa-95}).
Next, by the product property, we conclude that
\begin{equation*}
\mathrm{deg}_{\mathrm{B}}(\tilde{\eta},\widetilde{\Omega},0)
=
\mathrm{deg}_{\mathrm{B}}(\eta_{n},\widetilde{\Omega}_{1},0) \cdot \prod_{i=1}^{n-1} \mathrm{deg}_{\mathrm{B}}(\eta_{i},\Omega_{i+1}\cap\ker L_{i+1},0)
\end{equation*}
and the thesis follows.
\end{proof}

\section{Continuation theorems}\label{section-3}

In this section, we investigate the periodic problem
\begin{equation}\label{main-syst}
\begin{cases}
\, x_{i}' = g_{i}(t,x_{i+1}), &i=1,\ldots,n-1,
\\
\, x_{n}' = h(t,x_{1},\ldots,x_{n}),
\\
\, x_{i}(0)=x_{i}(T), &i=1,\ldots,n.
\end{cases}
\end{equation}
From now on, we assume that
\begin{itemize}
\item $\mathcal{D}$ is an open subset of $\mathbb{R}^{nm} =\prod_{i=1}^{n}\mathbb{R}^{m}$ and, for $i=1,\ldots,n$,
\begin{equation*}
\mathcal{D}_{i} \coloneqq \bigl{\{} w\in \mathbb{R}^{m} \colon \exists \, (x_{1}, \ldots,x_{n})\in\mathcal{D} \text{ such that } x_{i}=w\bigr{\}} 
\end{equation*}
denotes the projection of $\mathcal{D}$ onto the $i$-th component;
\item $0\in\mathcal{D}_{i}$, for every $i=2,\ldots,n$;
\item $g_{i}\colon \mathopen{[}0,T\mathclose{]}\times \mathcal{D}_{i+1} \to \mathbb{R}^{m}$ is an $L^{1}$-Carath\'{e}odory function, for every $i=1,\ldots,n-1$;
\item $h\colon\mathopen{[}0,T\mathclose{]}\times \mathcal{D} \to \mathbb{R}^{m}$ is an $L^{1}$-Carath\'{e}odory function.
\end{itemize}

A solution of \eqref{main-syst} is a function $x=(x_{1},\ldots,x_{n})$ such that, for every $i=1,\ldots,n$, $x_{i}\colon \mathopen{[}0,T\mathclose{]} \to \mathbb{R}^{m}$ is absolutely continuous with $x_{i}(0)=x_{i}(T)$, and the differential equations in \eqref{main-syst} are satisfied for a.e.~$t\in\mathopen{[}0,T\mathclose{]}$.

Our first goal is to write problem \eqref{main-syst} as a coincidence equation of the form
\begin{equation*}
Lx = N(x),\quad x\in \mathrm{dom}\,L\cap \mathrm{dom}\,N,
\end{equation*}
as discussed in Section~\ref{section-2}. With this aim, for $i=1,\ldots,n$, we consider the Banach spaces
\begin{equation*}
X_{i}\coloneqq \mathcal{C}(\mathopen{[}0,T\mathclose{]},\mathbb{R}^{m}),
\qquad
Z_{i}\coloneqq L^{1}(\mathopen{[}0,T\mathclose{]},\mathbb{R}^{m}),
\end{equation*}
endowed with the norms
\begin{equation*}
\|x_{i}\|_{\infty} \coloneqq \max_{t\in \mathopen{[}0,T\mathclose{]}} |x_{i}(t)|, \qquad \|z_{i}\|_{L^{1}}\coloneqq \int_{0}^{T} |z_{i}(t)| \,\mathrm{d}t,
\end{equation*}
respectively. As a consequence,
\begin{equation*}
X = \prod_{i=1}^{n} X_{i} = \mathcal{C}(\mathopen{[}0,T\mathclose{]},\mathbb{R}^{mn}),
\qquad
Z = \prod_{i=1}^{n} Z_{i} = L^{1}(\mathopen{[}0,T\mathclose{]},\mathbb{R}^{mn}).
\end{equation*}

For $i=1,\ldots,n$, we define
\begin{equation*}
\mathrm{dom}\,L_{i} \coloneqq \bigl{\{} x_{i} \in X_{i} \colon \text{$x_{i}$ is absolutely continuous and $x_{i}(0)=x_{i}(T)$} \bigr{\}}
\end{equation*}
and the map $L_{i}\colon \mathrm{dom}\,L_{i} \to Z_{i}$ as
\begin{equation*}
(L_{i}x_{i})(t)\coloneqq x'_{i}(t), \quad t\in\mathopen{[}0,T\mathclose{]}.
\end{equation*}
Hence, $L_{i}$ is a Fredholm linear operator of index zero. Moreover,  $\ker L_{i}$ is given by the constant functions in $\mathbb{R}^{m}$ and
\begin{equation*}
\mathrm{Im}\,L_{i} = \biggl{\{} z_{i}\in Z_{i} \colon \int_{0}^{T} z_{i}(t)\,\mathrm{d}t = 0 \biggr{\}}.
\end{equation*}
Next, we set
\begin{align*}
&P_{i} \colon X_{i} \to X_{i}, \qquad P_{i}x_{i} \coloneqq \dfrac{1}{T}\int_{0}^{T} x_{i}(t)\, \mathrm{d}t,
\\
&Q_{i} \colon Z_{i} \to Z_{i}, \qquad Q_{i}z_{i} \coloneqq \dfrac{1}{T}\int_{0}^{T} z_{i}(t)\, \mathrm{d}t.
\end{align*}
Notice that $\ker P_{i}$ is made up of the continuous functions with zero mean value.
Furthermore, we define $K_{i} \colon \mathrm{Im}\,L_{i} \to \mathrm{dom}\,L_{i} \cap \ker P_{i}$ as the right inverse of $L_{i}$ (depending on $P_i$), which is the operator that to every function $z_{i}\in Z_{i}$ with $\int_{0}^{T} z_{i}(t)\,\mathrm{d}t =0$ associates the unique map $x_{i}$ such that
\begin{equation*}
x_{i}'= z_{i} \quad \text{ and } \; \int_{0}^{T} x_{i}(t)\,\mathrm{d}t = 0;
\end{equation*}
notice that $x_{i}(0)=x_{i}(T)$.
As linear orientation-preserving isomorphism $J_{i} \colon \mathrm{Im}\,Q_{i} \to \ker L_{i}$, we choose the identity map in $\mathbb{R}^{m}$. 

Let
\begin{equation*}
\mathrm{dom}\,N \coloneqq \bigl{\{} x\in X \colon \text{$x(t)\in\mathcal{D}$, for all $t\in\mathopen{[}0,T\mathclose{]}$} \bigr{\}}
\end{equation*}
and consider the Nemytskii operator $N \colon \mathrm{dom}\,N \to Z$ with components
\begin{equation*}
\begin{cases}
\, N_{i}(x_{i+1})(t)\coloneqq g_{i}(t,x_{i+1}(t)), & i=1,\ldots,n-1, 
\\
\, N_{n}(x_{1},\ldots,x_{n})(t) \coloneqq h(t,x_{1}(t),\ldots,x_{n}(t)),
\end{cases}
\qquad t\in\mathopen{[}0,T\mathclose{]}.
\end{equation*}
An application of the Ascoli--Arzel\`{a} theorem shows that $N$ is an $L$-completely continuous operator.

In addition, we define
\begin{align*}
g_{i}^{\#} \colon \mathcal{D}_{i+1} \to \mathbb{R}^{m},
\qquad
&g_{i}^{\#}(w)\coloneqq \dfrac{1}{T} \int_{0}^{T} g_{i}(t,w)\,\mathrm{d}t, 
\quad i=1,\ldots,n-1,
\\
h^{\#} \colon \mathcal{D}\to \mathbb{R}^{m},
\qquad
&h^{\#}(s_{1},\ldots,s_{n})\coloneqq \dfrac{1}{T} \int_{0}^{T} h(t,s_{1},\ldots,s_{n})\,\mathrm{d}t,
\end{align*}
and
\begin{equation*}
\ell\colon \mathcal{D} \to \mathbb{R}^{mn},
\qquad \ell(s_{1},\ldots,s_{n})\coloneqq-\bigl{(}g_{1}^{\#}(s_{2}),\ldots,g_{n-1}^{\#}(s_{n}),h^{\#}(s_{1},\ldots,s_{n})\bigr{)}.
\end{equation*}

Given the above setting, we can now state and prove our first continuation theorem.

\begin{theorem}\label{th-cont-1}
Let $\Omega$ be an open (possibly unbounded) set in $X=\mathcal{C}(\mathopen{[}0,T\mathclose{]},\mathbb{R}^{mn})$ with $\Omega\subseteq \mathrm{dom}\,N$; denote by $\Omega_i$ its projection onto $X_i$, for $i=2,\ldots n$. 
Suppose that
\begin{itemize}
\item[$(h_{1})$]
there exists a compact set $\mathcal{K}\subseteq \Omega$ containing all the solutions in $\Omega$ of
\begin{equation}\label{syst-1-lambda}
\begin{cases}
\, x_{i}' = g_{i}(t,x_{i+1}), &i=1,\ldots,n-1, \\
\, x_{n}' = \lambda h(t,x_{1},\ldots,x_{n}), \\
\, x_{i}(0)=x_{i}(T), &i=1,\ldots,n,
\end{cases}
\end{equation}
for every $\lambda\in\mathopen{]}0,1\mathclose{[}$;

\item[$(h_{2})$] for every $i=1,\ldots,n-1$, 
\begin{itemize}
\item[$\bullet$] $0\in\Omega_{i+1}$, 
\item[$\bullet$] $g_{i}(t,w)=0$ for every $t\in\mathopen{[}0,T\mathclose{]}$ if and only if $w=0$, 
\item[$\bullet$] $g_{i}^{\#}$ is injective; 
\end{itemize}

\item[$(h_{3})$] the set $\hat{h}^{-1}(0)\cap \widetilde{\Omega}_{1}$ is compact, where $\hat{h}(w)\coloneqq h^{\#}(w,0,\ldots,0)$ for $w\in\mathbb{R}^{m}$ and $\widetilde{\Omega}_{1} \coloneqq \{w\in\mathbb{R}^{m} \colon (w,0,\ldots,0)\in\Omega\}$;

\item[$(h_{4})$] $\mathrm{deg}_{\mathrm{B}}(\hat{h},\widetilde{\Omega}_{1},0) \neq0$.

\end{itemize}
Then, problem \eqref{main-syst} has a solution in $\Omega$.
\end{theorem}

\begin{proof}
Observing that $\ell = -JQN|_{\ker L}$, we first aim to apply Proposition~\ref{th-2.6-FeltrinZanolin}.
Hypothesis $(a_{1})$ of the proposition requires to prove that, for $i=1,\ldots,n-1$, if there exists a constant function $w\in\mathbb{R}^{m}$ such that \begin{equation*}
\int_{0}^{T} g_{i}(t,w)\,\mathrm{d}t=0,
\end{equation*}
then $g_{i}(t,w)=0$ for every $t\in\mathopen{[}0,T\mathclose{]}$, and this is a direct consequence of hypothesis $(h_{2})$.
Concerning $(a_{2})$, if for $\lambda=1$ there is no compact set in $\Omega$ containing all the solutions of \eqref{syst-1-lambda}, then the thesis of the theorem trivially holds, otherwise $(a_{2})$ straightforwardly follows from hypothesis $(h_{1})$.
Finally, to prove $(a_{3})$ we show that $\mathcal{S}_{0}\coloneqq\{x\in\Omega\cap \mathbb{R}^{mn} \colon \ell(x)=0\}$ is compact.
From $(h_{2})$ we have that $g_{i}^{\#}(w)=0$ if and only if $w=0$. Then, $x\in\mathcal{S}_{0}$ implies that $x=(w,0,\ldots,0)$ for a suitable $w\in\widetilde{\Omega}_{1}$. The compactness of $\mathcal{S}_{0}$ immediately follows from $(h_{3})$.
Therefore, Proposition~\ref{th-2.6-FeltrinZanolin} can be applied to obtain
\begin{equation*}
\mathrm{D}_{L}(L-N,\Omega) = \mathrm{deg}_{\mathrm{B}}(\ell,\Omega \cap \mathbb{R}^{mn},0).
\end{equation*}

Now, we conclude the proof by applying Proposition~\ref{lem-2.8-FeltrinZanolin}. Since assumptions $(a_{4})$, $(a_{5})$, $(a_{6})$ trivially hold true, then
\begin{align*}
\mathrm{deg}_{\mathrm{B}}(\ell,\Omega \cap \mathbb{R}^{m},0) 
&= (-1)^{m(n+1)} \, \mathrm{deg}_{\mathrm{B}}(-\hat{h},\widetilde{\Omega}_{1},0) \cdot \prod_{i=1}^{n-1} \mathrm{deg}_{\mathrm{B}}(-g_{i}^{\#},\Omega_{i+1} \cap \mathbb{R}^{m},0).
\\
&= (-1)^{m} \, \mathrm{deg}_{\mathrm{B}}(\hat{h},\widetilde{\Omega}_{1},0) \cdot \prod_{i=1}^{n-1} \mathrm{deg}_{\mathrm{B}}(g_{i}^{\#},\Omega_{i+1} \cap \mathbb{R}^{m},0).
\end{align*}

By $(h_{2})$, for all $i=1,\ldots,n-1$, since $g_{i}^{\#}$ is injective, $g_{i}^{\#}(0)=0$ and $0\in \Omega_{i+1}$, it holds that
\begin{equation*}
\mathrm{deg}_{\mathrm{B}}(g_{i}^{\#},\Omega_{i+1} \cap \mathbb{R}^{m},0)=\pm 1
\end{equation*}
(cf.~\cite[Theorem~7.4.5]{DiMa-21}).

Finally, $(h_{4})$ implies that $\mathrm{deg}_{\mathrm{B}}(\ell,\Omega \cap \mathbb{R}^{m},0)\neq0$ and hence $\mathrm{D}_{L}(L-N,\Omega)\neq0$. The thesis follows from the existence property of the degree.
\end{proof}

For the second continuation theorem, we consider an $L^{1}$-Carath\'{e}odory homotopy $\tilde{h}\colon\mathopen{[}0,T\mathclose{]}\times \mathcal{D}\times\mathopen{[}0,1\mathclose{]} \to \mathbb{R}^{m}$ between $h$ and an autonomous vector field $h_{0}$, that is, we assume that
\begin{equation*}
\tilde{h}(t,x_{1},\ldots,x_{n},0)=h_{0}(x_{1},\ldots,x_{n}),
\qquad
\tilde{h}(t,x_{1},\ldots,x_{n},1)=h(t,x_{1},\ldots,x_{n}).
\end{equation*}
The following holds true.

\begin{theorem}\label{th-cont-2}
Let $\Omega$ be an open (possibly unbounded) set in $\mathcal{C}(\mathopen{[}0,T\mathclose{]},\mathbb{R}^{mn})$ with $\Omega\subseteq \mathrm{dom}\,N$.
Assume $(h_{2})$ and suppose that
\begin{itemize}
\item[$(h_{1}')$]
there exists a compact set $\mathcal{K}\subseteq \Omega$ containing all the solutions in $\Omega$ of
\begin{equation}\label{syst-2-lambda}
\begin{cases}
\, x_{i}' = g_{i}(t,x_{i+1}), &i=1,\ldots,n-1, \\
\, x_{n}' = \tilde{h}(t,x_{1},\ldots,x_{n},\lambda), \\
\, x_{i}(0)=x_{i}(T), &i=1,\ldots,n,
\end{cases}
\end{equation}
for every $\lambda\in\mathopen{]}0,1\mathclose{[}$;

\item[$(h_{3}')$] the set $\hat{h}_{0}^{-1}(0)\cap \widetilde{\Omega}_{1}$ is compact, where $\hat{h}_{0}(w) \coloneqq h_{0}(w,0,\ldots,0)$ for $w\in\mathbb{R}^{m}$ and $\widetilde{\Omega}_{1} \coloneqq \{w\in\mathbb{R}^{m} \colon (w,0,\ldots,0)\in\Omega\}$;

\item[$(h_{4}')$] $\mathrm{deg}_{\mathrm{B}}(\hat{h}_{0},\widetilde{\Omega}_{1},0) \neq0$.

\end{itemize}
Then, problem \eqref{main-syst} has a solution in $\Omega$.
\end{theorem}

The proof follows the same steps of the proof of Theorem~\ref{th-cont-1} and we omit the details.

Since the classical continuation theorems reviewed in the introduction consider the case of an open and bounded set $\Omega$, we now give the following versions of our continuation theorems in that framework. The proofs are straightforward.

\begin{theorem}\label{th-cont-1-ob}
Let $\Omega$ be an open and bounded set in $\mathcal{C}(\mathopen{[}0,T\mathclose{]},\mathbb{R}^{mn})$ with $\overline{\Omega}\subseteq \mathrm{dom}\,N$.
Assume $(h_{2})$, $(h_{4})$ and
\begin{itemize}
\item
there is no solution of \eqref{syst-1-lambda} on $\partial\Omega$ for $\lambda\in\mathopen{]}0,1\mathclose{[}$;

\item $\hat{h}^{-1}(0)\cap \partial\widetilde{\Omega}_{1}=\emptyset$.
\end{itemize}
Then, problem \eqref{main-syst} has a solution in $\overline{\Omega}$.
\end{theorem}

\begin{theorem}\label{th-cont-2-ob}
Let $\Omega$ be an open and bounded set in $\mathcal{C}(\mathopen{[}0,T\mathclose{]},\mathbb{R}^{mn})$ with $\overline{\Omega}\subseteq\mathrm{dom}\,N$.
Assume $(h_{2})$, $(h_{4}')$ and
\begin{itemize}
\item
there is no solution of \eqref{syst-2-lambda} on $\partial\Omega$ for $\lambda\in\mathopen{]}0,1\mathclose{[}$;

\item $\hat{h}_{0}^{-1}(0)\cap \partial\widetilde{\Omega}_{1} = \emptyset$.
\end{itemize}
Then, problem \eqref{main-syst} has a solution in $\overline{\Omega}$.
\end{theorem}

\section{Applications to second-order systems}\label{section-4}

This section deals with the second-order periodic problem
\begin{equation}\label{main-pb-MM}
\begin{cases}
\, (\phi(t,x'))'=f(t,x,x'),
\\
\, x(0)=x(T),\quad x'(0)=x'(T),
\end{cases}
\end{equation}
where $f\colon\mathopen{[}0,T\mathclose{]}\times\mathbb{R}^{m}\times U \to \mathbb{R}^{m}$ is an $L^{1}$-Carath\'{e}odory function.
We assume the following conditions on $\phi$:
\begin{itemize}
\item[$(\phi_{1})$] $\phi \colon \mathopen{[}0,T\mathclose{]}\times U \to V$ is a continuous function, where $U$ and $V$ are open subsets of $\mathbb{R}^{m}$ containing $0$; 
\item[$(\phi_{2})$] $\phi(t,0)=0$, for every $t\in\mathopen{[}0,T\mathclose{]}$;
\item[$(\phi_{3})$] $\phi(0,s)=\phi(T,s)$, for every $s\in U$;
\item[$(\phi_{4})$] $\phi(t,\cdot)$ is a homeomorphism with $\phi(t,U)=V$, for every $t\in\mathopen{[}0,T\mathclose{]}$.
\end{itemize}
Given any $t\in\mathopen{[}0,T\mathclose{]}$, being $\phi(t,\cdot)$ invertible, we denote, with a little abuse of notation, its inverse as $\phi^{-1}(t,\cdot)$.
We suppose also that
\begin{itemize}
\item[$(\phi_{*})$] the map $\displaystyle V \ni s \mapsto \dfrac{1}{T} \int_{0}^{T} \phi^{-1}(t,s)\,\mathrm{d}t$ is injective.
\end{itemize} 

A \textit{solution} of system \eqref{main-pb-MM} is a continuously differentiable map $x\colon \mathopen{[}0,T\mathclose{]}\to \mathbb{R}^{m}$ such that $t\mapsto\phi(t,x'(t))$ is absolutely continuous, the first equality in \eqref{main-pb-MM} holds for a.e.~$t\in \mathopen{[}0,T\mathclose{]}$, and the periodic boundary conditions are satisfied.

Let $\mathcal{C}^{1}_{T}$ be the space of continuously differentiable functions $x \colon \mathopen{[}0,T\mathclose{]}\to\mathbb{R}^{m}$ with $x(0)=x(T)$ and $x'(0)=x'(T)$.

The following continuation theorem holds true.

\begin{theorem}\label{th-cont-MM-1}
Let $\mathcal{O}$ be an open (possibly unbounded) set in $\mathcal{C}^{1}_{T}$ with $x'(t)\in U$ for every $t\in\mathopen{[}0,T\mathclose{]}$ and every $x\in\mathcal{O}$.
Assume that
\begin{itemize}
\item[$(\textsc{a}_{1})$]
there exists a compact set $\mathcal{K}\subseteq \mathcal{O}$ containing all the solutions in $\mathcal{O}$ of
\begin{equation}\label{eq-4.1}
\begin{cases}
\, (\phi(t,x'))'=\lambda f(t,x,x'),
\\
\, x(0)=x(T),\quad x'(0)=x'(T),
\end{cases}
\end{equation}
for every $\lambda\in\mathopen{]}0,1\mathclose{[}$;

\item[$(\textsc{a}_{2})$] the set $\hat{f}^{-1}(0)\cap \mathcal{O}$ is compact, where 
\begin{equation*}
\hat{f}(w) \coloneqq \dfrac{1}{T} \int_{0}^{T}f(t,w,0)\,\mathrm{d}t, \quad w\in\mathbb{R}^{m};
\end{equation*}
\item[$(\textsc{a}_{3})$] $\mathrm{deg}_{\mathrm{B}}(\hat{f},\mathcal{O}\cap\mathbb{R}^{m},0) \neq0$.
\end{itemize}
Then, problem \eqref{main-pb-MM} has a solution in $\mathcal{O}$.
\end{theorem}

\begin{proof}
We introduce the first-order periodic problem
\begin{equation}\label{syst-1}
\begin{cases}
\, x_{1}'=\phi^{-1}(t,x_{2}),
\\
\, x_{2}'=h(t,x_{1},x_{2}),
\\
\, x_{1}(0)=x_{1}(T), \quad x_{2}(0)=x_{2}(T),
\end{cases}
\end{equation}
where $h(t,x_{1},x_{2})=f(t,x_{1},\phi^{-1}(t,x_{2}))$.
Notice that $x$ solves \eqref{main-pb-MM} if and only if $(x_{1},x_{2})=(x,\phi(\cdot,x'(\cdot)))$ solves \eqref{syst-1}. The equivalence of the second-order equation and the first-order system is straightforward, while the equivalence of the periodic boundary conditions follows from the equality $\phi(0,x'(0))=\phi(T,x'(T))$, which is valid thanks to assumption $(\phi_{3})$.

With the aim of applying Theorem~\ref{th-cont-1} to problem \eqref{syst-1} (with $n=2$ and $g_{1}=\phi^{-1}$), we consider the parameter-dependent system 
\begin{equation}\label{eq-4.4}
\begin{cases}
\, x_{1}'=\phi^{-1}(t,x_{2}),
\\
\, x_{2}'=\lambda h(t,x_{1},x_{2}),
\\
\, x_{1}(0)=x_{1}(T), \quad x_{2}(0)=x_{2}(T),
\end{cases}
\end{equation}
equivalent to \eqref{eq-4.1}.

As a first step, using the same notation introduced in Section~\ref{section-3}, we need to define an open set $\Omega$ with
\begin{equation*}
\Omega\subseteq\mathrm{dom}\,N = \bigl{\{} x\in \mathcal{C}(\mathopen{[}0,T\mathclose{]},\mathbb{R}^{2m}) \colon \text{$x(t)\in\mathcal{D}$, for all $t\in\mathopen{[}0,T\mathclose{]}$} \bigr{\}},
\end{equation*}
where $\mathcal{D}=\mathbb{R}^{m}\times V$ and $N$ is defined as
\begin{equation*}
N(x)(t)=N(x_1,x_2)(t)=(\phi^{-1}(t,x_{2}(t)), h(t,x_{1}(t),x_{2}(t))).
\end{equation*}
Let $\mathcal{K} \subseteq \mathcal{O}$ be the compact set given in condition $(\textsc{a}_{1})$.
For every $y\in\mathcal{K}$, let $\varepsilon_y>0$ be such that the open set
\begin{equation}
\label{U y epsilon} 
\mathcal{U}_{y} \coloneqq \bigl{\{} x \in \mathcal{C}^{1}_{T} \colon \|x-y\|_{\infty}<\varepsilon_y, \; \|x'-y'\|_{\infty}<\varepsilon_y\bigr{\}}
\end{equation}
is contained in $\mathcal{O}$. Observe that
\begin{equation}\label{eq-finite-cov}
\mathcal{K} \subseteq \bigcup_{y\in \mathcal{K}}\mathcal{U}_{y} \subseteq \mathcal{O}.
\end{equation}
By hypotheses $(\textsc{a}_{2})$ and $(\textsc{a}_{3})$, the set $\mathcal{S}_0\coloneqq\hat{f}^{-1}(0)\cap \mathcal{O}$ is compact and nonempty. Thus, for any $z\in \mathcal{S}_0$, which is a constant function, there exists a corresponding $\varepsilon_z>0$ such that the open set
\begin{equation}\label{U z epsilon} 
\mathcal{U}_{z} \coloneqq \bigl{\{} x \in \mathcal{C}^{1}_{T} \colon \|x-z\|_{\infty}<\varepsilon_z, \; \|x'\|_{\infty}<\varepsilon_z\bigr{\}}
\end{equation}
is contained in $\mathcal{O}$. It is easily seen that
\begin{equation}\label{S0_contained}
\mathcal{S}_0 \subseteq \bigcup_{z\in \mathcal{S}_0}\mathcal{U}_{z} \subseteq \mathcal{O}.
\end{equation}

Next, for every $y\in\mathcal{K}$, corresponding to $\mathcal{U}_{y}$ given by formula \eqref{U y epsilon} we define in the space $\mathcal{C}(\mathopen{[}0,T\mathclose{]},\mathbb{R}^{2m})=\mathcal{C}(\mathopen{[}0,T\mathclose{]},\mathbb{R}^{m})\times\mathcal{C}(\mathopen{[}0,T\mathclose{]},\mathbb{R}^{m})$ the set
\begin{equation*}
\Lambda_{y} \coloneqq  
\left\{
\begin{array}{rl}
& \|x_{1}-y\|_{\infty}<\varepsilon_y, \vspace{3pt}\\
(x_{1},x_{2}) \in \mathcal{C}(\mathopen{[}0,T\mathclose{]},\mathbb{R}^{2m}) \colon  
&  \text{$x_2(t)\in V$, for every $t\in\mathopen{[}0,T\mathclose{]}$,}\vspace{3pt}\\
& \|\phi^{-1}(\cdot,x_{2}(\cdot))-y'\|_{\infty} <\varepsilon_y 
\end{array}
\right\}.
\end{equation*}
Analogously, for every $z\in \mathcal{S}_{0}$, corresponding to $\mathcal{U}_{z}$ in \eqref{U z epsilon} we define 
\begin{equation*}
\Lambda_{z} \coloneqq  
\left\{
\begin{array}{rl}
& \|x_{1}-z\|_{\infty}<\varepsilon_z, \vspace{3pt}\\
(x_{1},x_{2}) \in \mathcal{C}(\mathopen{[}0,T\mathclose{]},\mathbb{R}^{2m}) \colon  
&  \text{$x_2(t)\in V$, for every $t\in\mathopen{[}0,T\mathclose{]}$,}\vspace{3pt}\\
& \|\phi^{-1}(\cdot,x_{2}(\cdot))\|_{\infty} <\varepsilon_z 
\end{array}
\right\}.
\end{equation*}
Let us point out that all the sets $\Lambda_{y}$ and $\Lambda_{z}$ are nonempty and open. In fact, consider $y\in \mathcal{K}$. By the definition of $\mathcal{O}$, we have $y'(t)\in U$ for each $t$; consequently $\phi(\cdot, y'(\cdot))$ is well defined and its image is contained in $V$. 
In addition, $\phi^{-1}(\cdot,\phi(\cdot, y'(\cdot)))$ coincides with $y'$ and thus $(y, \phi(\cdot, y'(\cdot))) \in\Lambda_{y}$. 
The openness of this set is straightforward since $U$ and $V$ are open. The argument showing that the sets $\Lambda_{z}$ are nonempty and open is the same. Moreover, all the sets $\Lambda_{y}$ and $\Lambda_{z}$ are contained in $\mathrm{dom}\,N$.

We can now define
\begin{equation}
\label{definition Omega}
\Omega \coloneqq \bigcup_{y\in \mathcal{K}} \Lambda_{y} \cup \bigcup_{z\in \mathcal{S}_0} \Lambda_{z}. 
\end{equation}
Notice that $\Omega$ is open and contained in $\mathrm{dom}\,N$.

We are now in a position to verify hypotheses $(h_{1})$, $(h_{2})$, $(h_{3})$, and $(h_{4})$ of Theorem~\ref{th-cont-1}.
To check condition $(h_{1})$, given $\mathcal{K}\subseteq\mathcal{O}$ as in $(\textsc{a}_{1})$, we define the set
\begin{equation*}
\hat{\mathcal{K}}\coloneqq \bigl{\{} (x_{1},x_{2})\in\mathcal{C}(\mathopen{[}0,T\mathclose{]},\mathbb{R}^{2m}) \colon x_{1} \in \mathcal{K}, \; x_{2}=\phi(\cdot,x_{1}'(\cdot)) \bigr{\}}
\end{equation*}
and observe that $\hat{\mathcal{K}}\subseteq \bigcup_{y\in \mathcal{K}}\, \Lambda_{y}\subseteq\Omega$, as a consequence of \eqref{eq-finite-cov} and the definition of $\Omega$. In addition,  $\hat{\mathcal{K}}$ is compact being the image of $\mathcal{K}$ by the continuous map 
\begin{equation*}
G\colon \mathcal{C}^{1}_{T} \to \mathcal{C}(\mathopen{[}0,T\mathclose{]},\mathbb{R}^{2m}),
\quad
G(x)=(x,\phi(\cdot,x')).
\end{equation*}
Moreover, by $(\textsc{a}_{1})$, $\hat{\mathcal{K}}$ contains all the solutions of \eqref{eq-4.4}, for all $\lambda\in\mathopen{]}0,1\mathclose{[}$, recalling that problems \eqref{eq-4.1} and \eqref{eq-4.4} are equivalent. Then, condition $(h_{1})$ holds.

Concerning condition $(h_{2})$, we first observe that $0\in\Omega_{2}$, since $(z,0)\in\Omega$ for any $z\in \mathcal{S}_0$, by construction.
Moreover, for every $t\in\mathopen{[}0,T\mathclose{]}$, $\phi^{-1}(t,s)=0$  if and only if $s=0$, by $(\phi_{2})$ and $(\phi_{4})$. The injectivity is assumed in $(\phi_{*})$. Then, condition $(h_2)$ holds as well. 

To verify $(h_{3})$ and $(h_{4})$ we proceed as follows. 
Let $\hat{h}(w)\coloneqq\frac{1}{T}\int_0^T h(t,w,0)\,\mathrm{d}t$, for $w\in\mathbb{R}^{m}$, and $\widetilde{\Omega}_{1} \coloneqq \{w\in\mathbb{R}^{m} \colon (w,0)\in\Omega\}$ be as in hypothesis $(h_{3})$. From $(\phi_{2})$, we have that $\hat{f}$ and $\hat{h}$ coincide.

We show first that
\begin{equation*}
\hat{h}^{-1}(0)\cap \widetilde{\Omega}_{1} \subseteq \mathcal{S}_{0}.
\end{equation*} 
In fact, take $w\in \widetilde{\Omega}_{1}$ such that $\hat h(w)=0$. 
We have that $(w,0)\in \Omega$, so there exists either $y\in \mathcal{K}$ such that $(w,0)\in \Lambda_y$, or $z \in \mathcal{S}_{0}$  such that $(w,0) \in \Lambda_{z}$. In the first case, we have $w\in \mathcal{U}_{y}$, in the second one $w\in \mathcal{U}_{z}$.
In both cases, $w\in \mathcal{O}$ because all the sets  $\mathcal{U}_{y}$ are $\mathcal{U}_{z}$ are contained in $\mathcal{O}$. 
Since, by definition, $\mathcal{S}_0=\hat{f}^{-1}(0)\cap \mathcal{O}$ and recalling that $\hat{f}=\hat{h}$, then $w \in \mathcal{S}_0$. 

Next, we claim that 
\begin{equation}\label{inclusion_1}
\bigcup_{z\in \mathcal{S}_{0}}\,\mathcal{U}_{z} \cap\mathbb{R}^{m}\subseteq \widetilde{\Omega}_{1}.
\end{equation}
To see this, fix $w\in \bigcup_{z\in \mathcal{S}_{0}}\,\mathcal{U}_{z} \cap\mathbb{R}^{m}$ and let 
$\bar z\in \mathcal{S}_{0}$ be such that $w \in \mathcal{U}_{\bar z}$. 
This implies that $(w,0)\in \Lambda_{\bar z}$, which is contained in $\Omega$ by \eqref{definition Omega}. 
Therefore, $w\in \widetilde{\Omega}_{1}$ and the claim follows.
As a byproduct of \eqref{inclusion_1}, we have that $\mathcal{S}_0 \subseteq \widetilde{\Omega}_{1}$. Thus, $\hat{h}^{-1}(0)\cap \widetilde{\Omega}_{1}=\mathcal{S}_0$ which is compact by $(\textsc{a}_{2})$, and so $(h_{3})$ is verified.

At last, we have to show that $(h_{4})$ holds, that is $\mathrm{deg}_{\mathrm{B}}(\hat{h},\widetilde{\Omega}_{1},0) \neq0$.
Recalling \eqref{inclusion_1}, \eqref{S0_contained}, and that $\mathcal{S}_{0}$ is the set of solutions of $\hat{f}(w)=0$ in $\mathcal{O}\cap\mathbb{R}^{m}$, we can apply twice the excision property of the degree obtaining
\begin{equation*}
\mathrm{deg}_{\mathrm{B}}(\hat{h},\widetilde{\Omega}_{1},0) 
=
\mathrm{deg}_{\mathrm{B}}(\hat{f},\bigcup_{z\in \mathcal{S}_{0}}\,\mathcal{U}_{z} \cap\mathbb{R}^{m},0)
=
\mathrm{deg}_{\mathrm{B}}(\hat{f},\mathcal{O}\cap\mathbb{R}^{m},0).
\end{equation*}
Consequently, $(h_{4})$ follows form $(\textsc{a}_{3})$. 

Therefore, Theorem~\ref{th-cont-1} can be applied and we deduce the existence of a solution $(x_{1},x_{2})$ of \eqref{syst-1} in $\Omega$. Finally, setting $x\coloneqq x_{1}$, we have that $x$ solves \eqref{main-pb-MM}, as observed above. The proof is thus complete.
\end{proof}

For the second continuation theorem, we consider an $L^{1}$-Carath\'{e}odory map $\tilde{f}\colon\mathopen{[}0,T\mathclose{]}\times \mathcal{D}\times\mathopen{[}0,1\mathclose{]} \to \mathbb{R}^{m}$ such that
\begin{equation*}
\tilde{f}(t,x_{1},x_{2},0)=f_{0}(x_{1},x_{2}),
\qquad
\tilde{f}(t,x_{1},x_{2},1)=f(t,x_{1},x_{2}).
\end{equation*}
The following holds true. The proof is analogous to the one of Theorem~\ref{th-cont-MM-1} and is therefore omitted.

\begin{theorem}\label{th-cont-MM-2}
Let $\mathcal{O}$ be an open (possibly unbounded) set in $\mathcal{C}^{1}_{T}$ with $x'(t)\in U$ for every $t\in\mathopen{[}0,T\mathclose{]}$ and every $x\in\mathcal{O}$.
Assume that
\begin{itemize}
\item[$(\textsc{a}_{1}')$]
there exists a compact set $\mathcal{K}\subseteq \mathcal{O}$ containing all the solutions in $\mathcal{O}$ of
\begin{equation*}
\begin{cases}
\, (\phi(t,x'))' = \tilde{f}(t,x,x',\lambda), 
\\
\, x(0)=x(T),\quad x'(0)=x'(T),
\end{cases}
\end{equation*}
for every $\lambda\in\mathopen{]}0,1\mathclose{[}$;
\item[$(\textsc{a}_{2}')$] the set $\hat{f}_{0}^{-1}(0)\cap \mathcal{O}$ is compact, where $\hat{f}_{0}(w) \coloneqq f_{0}(w,0)$ for $w\in\mathbb{R}^{m}$;
\item[$(\textsc{a}_{3}')$] $\mathrm{deg}_{\mathrm{B}}(\hat{f}_{0},\mathcal{O}\cap\mathbb{R}^{m},0) \neq0$.
\end{itemize}
Then, problem \eqref{main-pb-MM} has a solution in $\mathcal{O}$.
\end{theorem}

\begin{remark}[The hypotheses on $\phi$]\label{rem-4.1}
Firstly, we point out that, if $\phi$ does not depend on time (i.e., $\phi(t,s)=\phi(s)$), all the hypotheses on $\phi$ reduce to the following one:
\begin{itemize}
\item $\phi \colon U \to V$ is a homeomorphism with $\phi(U)=V$, where $U$ and $V$ are open subsets of $\mathbb{R}^{m}$ containing $0$, and $\phi(0)=0$.
\end{itemize}
Indeed, hypotheses $(\phi_{1})$--$(\phi_{4})$ and $(\phi_{*})$ are straightforward consequence of the above assumption.
This ``autonomous'' framework improves the one in \cite{MaMa-98}. Indeed, in the case $U=V=\mathbb{R}^{m}$ considered in \cite{MaMa-98}, when $m\geq2$ it is possible to propose examples of homeomorphisms which do not satisfy the monotonicity and coercivity conditions $(\textsc{h}_{1})$ and $(\textsc{h}_{2})$ introduced therein (recalled in our introduction).
For instance, for $m=2$, one can consider the map
\begin{equation}\label{rem-ex-1}
\phi \colon \mathbb{R}^{2} \to \mathbb{R}^{2}, \quad \phi (s_{1},s_{2}) = (-s_{2},-s_{1}),
\end{equation}
which is a homeomorphism of $\mathbb{R}^{2}$, but does not satisfy $(\textsc{h}_{1})$ and $(\textsc{h}_{2})$ (see also \cite[Remark~3.12]{FeZa-17}).
Moreover, our framework improves also \cite{FeZa-17}, since in our case $U$ and $V$ are not imposed to be the whole $\mathbb{R}^{m}$. Accordingly, we can for instance consider the mean curvature operator and the Minkowski one:
\begin{align*}
&\phi \colon \mathbb{R}^{m} \to B_{\mathbb{R}^{m}}(0,1), \quad \phi(s) = \dfrac{s}{\sqrt{1+|s|^{2}}},
\\
&\phi \colon B_{\mathbb{R}^{m}}(0,1) \to \mathbb{R}^{m}, \quad \phi(s) = \dfrac{s}{\sqrt{1-|s|^{2}}},
\end{align*}
or the Perona--Malik operator, which introduces anisotropic diffusion in image processing, the hyperelastic tangent flow operator, which smooths sharp transitions through saturation effects, the congestion-type cost gradient operator, modeling limited mobility in high-density regions as observed in transport and crowd dynamics, and many other nonlinear differential operators arising in geometry, physics, biology, and engineering.

In the case of time-dependent differential operators, some remarks are in order regarding hypothesis $(\phi_{*})$.
First, we observe that hypotheses $(\textsc{h}_{1}')$ and $(\textsc{h}_{2}')$ in \cite{GHMMT-24} are significantly stronger. Indeed, as a consequence of the sole $(\textsc{h}_{1}')$, we have
\begin{equation*}
\dfrac{1}{T} \int_{0}^{T} \langle \phi^{-1}(t,s_{1})-\phi^{-1}(t,s_{2}), s_{1}-s_{2} \rangle  \,\mathrm{d}t >0, \quad \text{for every $s_{1},s_{2}\in\mathbb{R}^{m}$ with $s_{1}\neq s_{2}$,}
\end{equation*}
which implies that $(\phi_{*})$ holds true. Moreover, by considering operators of the form
\begin{equation}\label{rem-ex-2}
\phi(t,s) = \eta(t) \varphi(s),
\end{equation}
where $\eta \colon \mathopen{[}0,T\mathclose{]} \to \mathopen{]}0,+\infty\mathclose{[}$ and $\varphi \colon \mathbb{R}^{m} \to \mathbb{R}^{m}$, it is easy to adapt the above counterexample \eqref{rem-ex-1} to show that there are maps $\phi$ satisfying $(\phi_{1})$--$(\phi_{4})$ and $(\phi_{*})$, but not $(\textsc{h}_{1}')$ and $(\textsc{h}_{2}')$.

Furthermore, we can introduce a class of operators different from that considered in \cite{GHMMT-24}. Precisely, in addition to $(\phi_{1})$--$(\phi_{4})$ and $(\phi_{*})$, we assume that the inverse operator $\phi^{-1}$ is of the form
\begin{itemize}
\item[$(\phi_{\#})$]  $\phi^{-1}(t,s) = \sigma(t,s) \tau(s)$, where $\sigma \colon \mathbb{R} \times V \to \mathbb{R}$ and $\tau \colon V \to \mathbb{R}^{m}$ are continuous functions.
\end{itemize}
As an illustrative example, the $p(t)$-Laplacian operator $\phi(t,s)=|s|^{p(t)-2}s$ (with inverse $\phi^{-1}(t,s)=|s|^{q(t)-2}s$ for an appropriate $q$) belongs to this class of operators, provided that suitable conditions on $p$ are satisfied. Clearly, functions of the form \eqref{rem-ex-2}, which do not satisfy $(\textsc{h}_{1}')$ and $(\textsc{h}_{2}')$,
may belong to this class.

To show that, under assumption $(\phi_{\#})$, condition $(\phi_{*})$ is satisfied, we employ the mean value theorem for integrals to deduce that 
\begin{equation*}
\dfrac{1}{T} \int_{0}^{T} \phi^{-1}(t,s)\,\mathrm{d}t = \tau(s) \dfrac{1}{T} \int_{0}^{T} \sigma(t,s)\,\mathrm{d}t = \tau(s) \sigma(\tilde{t},s) = \phi^{-1}(\tilde{t},s),
\end{equation*}
for some $\tilde{t}\in\mathopen{[}0,T\mathclose{]}$; thus the injectivity directly follows from the fact that $\phi^{-1}(\tilde{t},\cdot)$ is a homeomorphism.

We finally observe that, again by the mean value theorem for integrals, if $m=1$ condition $(\phi_{1})$ directly implies $(\phi_{*})$. This is not the case in higher dimension (i.e., $m\geq2$), as one can see by considering the homeomorphism
\begin{equation*}
\phi(t,s) = 
\begin{pmatrix}
\cos\bigl{(}\frac{2\pi}{T} t\bigr{)} & -\sin\bigl{(}\frac{2\pi}{T} t\bigr{)}
\vspace{3pt} \\
\sin\bigl{(}\frac{2\pi}{T} t\bigr{)} & \cos\bigl{(}\frac{2\pi}{T} t\bigr{)}
\end{pmatrix}
s,
\qquad (t,s)\in\mathbb{R}\times\mathbb{R}^{2},
\end{equation*}
which obviously satisfies $(\phi_{1})$--$(\phi_{4})$, but the mean value of the rotation matrix gives an identically zero function which is clearly not injective.
\hfill$\lhd$
\end{remark}
 
\bibliographystyle{elsart-num-sort}
\bibliography{BeFe-biblio}

\end{document}